\documentclass[11pt,a4paper]{article}
\usepackage{a4,amssymb,amsmath,amsthm,url,color,lineno}
\usepackage{bezier,amsfonts,amssymb,graphicx,amsthm,url}%5,mathtools}
\usepackage[english]{babel}
\title{Equating two maximum degrees}
\date{}
\begin{document}
\newtheorem{theorem}{Theorem}[section]
\newtheorem{definition}{Definition}[section]
\newtheorem{proposition}[theorem]{Proposition}
\newtheorem{corollary}[theorem]{Corollary}
\newtheorem{lemma}[theorem]{Lemma}
\newtheorem{conjecture}[theorem]{Conjecture}
\newcommand{\diff}{\mbox{\rm diff}}
%\newtheorem{Ex}{$\rhd$ Example}[section]
%\linenumbers
\DeclareGraphicsExtensions{.pdf,.png,.jpg}

\author{Yair Caro \\ Department of Mathematics\\ University of Haifa-Oranim \\ Israel \and Josef  Lauri\\ Department of Mathematics \\ University of Malta
\\ Malta \and Christina Zarb \\Department of Mathematics \\University of Malta \\Malta }

\maketitle

\begin{abstract}

Given a graph $G$, we would like to find (if it exists)  the largest induced subgraph $H$ in which there are at least $k$ vertices realizing the maximum degree of $H$.  This problem was first posed by Caro and Yuster. They proved, for example, that  for every graph $G$ on $n$ vertices we can guarantee, for $k = 2$,  such an induced subgraph $H$  by deleting at most $2\sqrt{n}$ vertices, but the question if $2\sqrt{n}$  is best  possible remains open.
 
Among the results obtained in this paper we prove that:
 \begin{enumerate}
\item{For every graph $G$ on $n \geq 4$  vertices we can delete at most  $\lceil \frac{- 3 + \sqrt{ 8n- 15}}{2 } \rceil$  vertices  to get an induced subgraph $H$ with at least two vertices realizing $\Delta(H)$, and this bound is sharp, solving the problems left open by Caro and Yuster.}
 \item{For every graph $G$ with maximum  degree $\Delta \geq 1$ we can delete at most $\lceil \frac{ -3 + \sqrt{8\Delta +1}}{2 } \rceil$  vertices   to get an induced subgraph $H$ with at least two vertices realizing $\Delta(H)$, and this bound is sharp.}
\item{Every graph $G$ with $\Delta(G) \leq 2$  and least $2k  - 1$ vertices (respectively  $2k - 2$  vertices if k is even) contains an induced subgraph $H$ in which at least $k$ vertices realise $\Delta(H)$,  and these bound are sharp.}
\end{enumerate}
 
\end{abstract}

\section{Introduction}

A well-known elementary exercise in graph theory states that every (simple) graph on at least two vertices has two vertices with the same degree.
Motivated by this fact, Caro and West \cite{caro2009repetition} formally defined the repetition number of a graph $G$, $rep(G)$, to be the maximum multiplicity in the list (degree sequence) of the vertex degrees.

Various research was done concerning the repetition number or repetitions in the degree sequence. Here we mention some of these directions.
\begin{enumerate}
\item{The connection  between the independence number and  $K_r$-free graphs with given repetition number \cite{bollobas1996degree,bollobas1997independent,erdHos1995degree}.}
\item{Hypergraph irregularity - the existence of  $r$-uniform hypergraphs ($r \geq 3$)  with no repeated degrees \cite{balister2016random,gyarfas1992irregularity}.}
\item{Ramsey type problems with repeated degrees  \cite{albertson1991ramsey} and \cite{erdos1993ramsey}. }
\item{Regular independent sets  --- vertices of the same degree forming an independent set \cite{albertson1994lower,albertson79con,caro2016regular}.}
\item{ Forcing $k$-repetition anywhere in the degree sequence \cite{caro2014forcing}}
\item{ Forcing $k$-repetition of the maximum degree \cite{caro2010large}.}
 \end{enumerate}

In this paper we shall focus on the following problem first stated in \cite{caro2010large}.   For a graph $G$ and an integer $k \geq 2$  let $f_k(G)$  denote the minimum number of vertices we have to delete from $G$ in order to get an induced subgraph $H$ in which there are at least $k$ vertices that attain the maximum degree $\Delta(H)$, of $H$,  or otherwise $|H| < k$, where  as  usual,  following the notation of  \cite{westintro},  $|G| = n$  is the number of vertices of $G$, $\Delta(G)$ is the maximum degree of $G$ and a vertex of degree $t$ is called a $t$-vertex.   In the case $ k = 2$  we use the abbreviation $f(G)$ instead of $f_2(G)$.    We define $f(n,k) = \max \{ f_k(G): |G| \leq n  \}$  and $g(\Delta,k) = \max \{ f_k(G) :  \Delta(G) \leq \Delta \}$.

Clearly there are graphs in which we cannot equate $k$ degrees let alone $k$ maximum degrees. A simple example is the star $K_{1,k-1}$  for $k \geq 3$ having $k$ vertices and by definition $f_ j(K_{1,k-1}) =  1$ for $j \geq2$.
 
However, it is trivial  that in every graph $G$ on at least $R(k,k)$ vertices (where $R(k,k)$ is the diagonal Ramsey number), we can equate $k$ maximum degrees.   We call a graph $G$ in which  (by deleting vertices) we can equate $k$ maximum degrees a \emph{k-feasible} graph.  So  of interest is the following function \[h(\Delta,k)  = \max \{ |G| : \Delta(G) \leq \Delta \mbox{ and $G$  is not $k$-feasible}\}.\]

Caro and Yuster \cite{caro2010large} conjectured  that  for every $k \geq 2$  there exists a constant  $c(k)$ such that  $f(n,k) \leq c(k) \sqrt{n}$ and proved the conjecture  for $k = 2$ with  $c(2) =  2$ and $k =  3$ with $c(3) = 43$.  For $k \geq 4$ the conjecture is still open.  The question whether  $c(2) = 2$  and $c(3)=43$  are best possible also remains  open.

Our main purpose in this paper is to show:
 \begin{enumerate}
\item{$f(G)$  can be computed exactly in polynomial time $O(n^2)$.}
 \item{for  $\Delta \geq 1$,   $g(\Delta,2)) \leq  \lceil ( \frac{- 3 +  \sqrt{8 \Delta+1}}{2} \rceil$ and this bound  is sharp.}
\item{for $n \geq 4$, $f(n,2)  \leq   \lceil ( \frac{- 3 +  \sqrt{8n-15}}{2} \rceil$ and this bound  is sharp. Hence in particular $f(2) =   \sqrt{2}$,  solving the problem left open in \cite{caro2010large}.}
\item{for a forest $F$  on $n$ vertices, $f_k(F) \leq (2k-1)n^{\frac{1}{3}}$.}
\item{$g(1,k) = \lfloor \frac{ k-1}{2} \rfloor$,  $g(2,k) = k-1$, thus determining exactly $g(\Delta,k)$  for $\Delta = 1,2$.}
\item{$h(0,k) = k-1$,  $h(1,k) =  \lfloor \frac{ k}{2} \rfloor + 2 \lfloor \frac{k-1}{2} \rfloor$,  $h(2,k) = 2k-2$  for odd $ k \geq 3$,  $h(2,k) = 2k-3$ for even $k \geq 2$.}
\end{enumerate}
 
 The paper is organized as follows :
 
In section 2  we cover the complexity issue of computing $f(G)$,  as well as the sharp upper-bounds for $g(\Delta,2)$ and $f(n,2)$.  In section 3  we consider upper-bounds for $f(F)$  and $f_k(F)$  where $F$ is a forest.  In  section 4 we prove exact results about $g(\Delta,k)$ and $h(\Delta,k)$  for $\Delta= 0,1,2$.  Finally, in section 5  we shall collect open problems and conjectures that deserve further exploration.

\section{Determination of exact upper bounds for $f(G)$ in terms of $\Delta(G)$ and $|G| = n$.  }

We first need  a definition and two lemmas:
 
We call  $B  \subset V(G)$, a set of vertices in a graph $G$,  a \emph{2-equating set}  if in the induced subgraph $H$ on $V(G) \backslash   B$, there are at least two vertices that realise $\Delta(H)$.  We say that $B$ is a 2-equating set  which realises $f(G)$ if $B$ has the  minimum cardinality among all 2-equating sets of $G$.

Let the degree sequence of the graph $G$ on $n$ vertices be $\Delta= d_1 \geq d_2 \geq d_3 \geq \ldots \geq d_n= \delta$ so that $\Delta$ is the maximum degree and $\delta$ the minimum degree.  We define $\diff(G)=d_1-d_2$.
 
\begin{lemma} \label{lemma0}
Let $G$ be a graph on $n  \geq 2$ vertices  with  degree sequence $d_1 \geq d_2 \geq \ldots \geq d_n$,  with $\deg(v) = d_1$ and $\deg(u) = d_2$.   Then $f(G)  \leq  d_1 - d_2=\diff(G)$.
 \end{lemma}

\begin{proof}
 
If $d_1 = d_2$ then clearly $f(G) = 0$.  So let  $ \diff(G)=d_1 -d _2   \geq 1$. But then there is at least one set $B$ of neighbours of $v$ of size $\diff(G)$, none of which are adjacent to $u$, and clearly  $f(G) \leq |B| = \diff(G)$.

%$| N(v)\backslash N(u)| = t$  and we can delete these $t$ vertices to equate two maximum degrees as required.
\end{proof}

\begin{lemma}  \label{lemma_1}
 
Let $G$ be a graph on $n \geq 2$ vertices,  with degree sequence $d_1 \geq d_2 \geq \ldots \geq d_n$, with $\deg(v) = d_1$.  
Then either   $f(G) = \diff(G)$,  or  $v$  must be in every minimal 2-equating set of $G$.
\end{lemma}
 
\begin{proof} 
Suppose $f(G)$  is not realised by $\diff(G)$ ( which includes the case $d_1 = d_2$). Then $f(G)  < \diff(G)$, and we may assume that $v$ is the unique vertex in $G$ of degree $d_1$.
 
Let $f(G)$ be realised by some induced subgraph $H$,  and $B =  V(G) - V(H)$  is a minimum 2-equating set for $G$. Assume for the contrary that $v \in H$ (then  $v$ is not a member of $B$).
 
Since $v \in H$  but $f(G)$ is not realised by $\diff(G)$  then either $v$ is not of maximum degree in $H$  in which case at least  $\diff(G) +1$ vertices among the neighbours of $v$ must be deleted  contradicting  $|B| = f(G)  < \diff(G)$, or  $v$ is of maximum degree in $H$ and still at least $\diff(G)$  among its neighbours must be deleted and again  $f(G) =   |B| \geq \diff(G)$ contradicting $f(G) < \diff(G)$. 
\end{proof}

\begin{lemma} \label{lemma_2}
 Let $G$ be a graph on at least $n \geq 2$ vertices.  Suppose that  $f(G) \not =  \diff(G)$,  then  $f(G) = 1 + f(G -v)$, where  $v$ is the single vertex of maximum degree in $G$.
 \end{lemma}

\begin{proof}
 
Since $f(G) \not =  \diff(G)$  it follows  that  there is a single  vertex  $v$ of maximum degree in $G$  and also from Lemma \ref {lemma_1} we infer that $v$ must be in any minimal 2-equating set of $G$.
 
Let  $G_1 = G \backslash \{ v\}$  and let $B$ be a minimal 2-equating set for $G_1$, namely  $f(G_1)  = |B|$. Then clearly  $B \cup  \{ v\}$  is a 2-equating set for $G$ hence $f(G) \leq 1+ f(G_1)$.
 
On the other hand let $B$ be a minimum 2-equating set for $G$. Then by assumption and Lemma \ref{lemma_1} $v  \in   B$.  Set $B_1 =  B \backslash \{ v \}$.  Clearly  $B_1$ is a 2-equating set of $G_1$ hence $f(G_1) \leq |B_1| = |B| - 1 =  f(G) - 1$ which gives  $f(G_1) +1 \leq f(G)$.
 
Hence combining  both inequalities    we get $f(G) = f(G \backslash \{ v\} )  +1$.   
 \end{proof}
 
\begin{theorem} \label{theorem_1}
Let $G$ be a graph on $n \geq 2$ vertices,  then  \[f(G)= \min\{\diff(G_j)+j: j=0 \ldots n-2 \},\]where  $G_ { j+1}$  is  obtained from  $G_ j$ by deleting  the vertex $v_{1,j}$ of the maximum degree $d_{1,j}$ from $G_ j$ (where $G_0$ is taken to be $G$), and $d_{2,j}$ is the second largest degree in $G_ j$.

Moreover $f(G)$  can be determined in time $O(n^2)$.
\end{theorem}
 
\begin{proof}
 
By Lemma \ref{lemma_1},  either $f(G)  = \diff(G)$  or   $v_{1,0}$  must be deleted to obtain $G_1$ and in this case by Lemma \ref{lemma_2},  $f(G) = f(G_1) +1$.
 
Hence $f(G) =  \min\{ \diff(G), f(G_1) +1\}$.  
 
Now  again either $f(G_1) =  \diff(G_1)$,  or  by Lemma \ref{lemma_1} and Lemma \ref{lemma_2} the maximum degree in $G_1$ must be deleted to obtain $G_2$ and then  $f(G_1) = f(G_2) +1$.
 
Hence  $f(G) = \min\{\diff(G),  \diff(G_1) +1, f(G_2 )+2 \}$.   
 
We continue this process  until  for some first $j$,  $\diff(G_j) = 0$   and there we stop having two vertices realizing the maximum degree of $G_ j$ ( the later steps will always give a larger value then $\diff(G_j) +j =  j $).
 
Each step is forced by Lemma \ref{lemma_1} and Lemma \ref{lemma_2} , hence  \[f(G)= \min\{\diff(G_j)+j: j=0 \ldots n-2 \}.\]
 
Now   in each iteration we have to construct $G_ j$ from $G_ {j-1}$  by deleting the maximum degree $v_{1,j-1}$  from $G_{ j-1}$  
and compute $d_{1,j}$ and $d_{2,j}$  which can be done in $O(n)$ time running over the new degree sequence of $G_ j$ that can be computed from the degree sequence of $G_{j-1}$ by updating $d_{1,j}$ values in it .
 
 So the total running time for the algorithm is $O(n^2 + e(G) )  = O(n^2 )$.
\end{proof}

\begin{theorem} \label{theoremdelta}
Let $G$ be a graph on $n \geq 2$ vertices  with maximum degree $\Delta$, and $t \geq 1$ be an integer. 
\begin{enumerate}
\item{If $0 \leq \Delta \leq 1$, then $f(G)=0$.}
\item{If $\binom{t +1}{2} +  1 \leq \Delta \leq  \binom{t +2}{2}$, then  $f(G) \leq  t$,   and this bound is sharp for every $\Delta$ in the range.}
\item{For $\Delta \geq 1$, $f(G) \leq \left \lceil \frac{-3+\sqrt{8\Delta+1}}{2} \right \rceil$.}
\end{enumerate}

\end{theorem}
\begin{proof}

Clearly, if $0 \leq \Delta \leq 1$ and $n \geq 2$, $f(G)=0$.

For (ii), we use induction on $t$.  For $t = 1$,   $2 \leq \Delta \leq 3$.  If there is one vertex  $v$ of degree $\Delta=2$.  Removing $v$ clearly  leaves at least two vertices of maximum degree equal to one or zero, and hence $f(G)=1$.  If there is a vertex $v$ of degree $\Delta=3$ and a vertex $u$ of degree 2, then by Lemma \ref{lemma0}, $f(G) \leq 3-2=1$.  Otherwise, all other vertices have degree 0 or 1 and deleting $v$ leaves at least two vertices of maximum degree equal to one or zero, and $f(G)=1$.

So assume statement is true for $t -1$ and we shall prove it is true for $t$.
 
By assumption,  $\binom{t +1}{2} +  1 \leq \Delta \leq  \binom{t +2}{2}$. Let $v$ be a vertex of maximum degree and $u$ a vertex with the second largest degree ---  clearly  \[\deg(u) \leq \deg(v) \leq \binom{t +2}{2}.\]
 
%If $ \deg(v) \leq \binom{t+1}{2}$, then by induction $f(G)  \leq t - 1$ and we are done. So we may assume  $\binom{t+1}{2}  + 1 %\leq \deg(v) \leq \binom{t+2}{2}$.
 
Now consider $\deg(u)$. 
\begin{enumerate}
\item{if  $\deg(u) \leq \binom{t+1}{2}$ we drop $v$ to get $G \backslash \{ v \} = H$  where $\Delta(H)\leq \binom{t+1}{2}$, and by induction  $f(G) \leq f(H) +1 \leq t - 1 +1  = t$  and we are done.}
\item{if $\deg(u) \geq \binom{t+1}{2}  +1$  then clearly  \[\diff(G) = \deg(v) - \deg(u) \leq \binom{t+2}{2}  -  \binom{t+1}{2}  - 1  =  t ,\]  hence by Theorem \ref{theorem_1} $f(G)  \leq \diff(G) \leq t$  and we are done.}
\end{enumerate}

Sharpness:   consider the sequence  $a_ j = \binom{ j +1}{2}  +1$  i.e.  $a_1 = 2$,  $a_2 = 4$,  $a_3 = 7$  etc. and let  $\Delta=\binom{ t +1}{2}+ j$,  $j = 1, \ldots, t +1$.  For example, if $t = 4$, $\Delta= 11,12,13,14,15$.  
 
Consider the graph  $G_{\Delta}$  consisting of the stars  $K_{1,a_j}$ for $j = 1, \ldots, t-1$  and a ``big star"  $K_{1,\Delta}$.
 
Suppose for example $\Delta= 13$ , which is the case $t = 4 $, since $ \binom{ t+1}{2}=\binom{5}{2} < 13  < \binom{6}{2}=\binom{t+2}{2}.$   The sequence of stars  we choose involves $a_1 ,a_2 ,a_3$  and $\Delta$, that is $ K_{1,2} \cup  K_{1,4}  \cup  K_{1,7} \cup  K_{1,13}$, and this realises $f(G) = 4$ as required.  The validity of this construction is a simple application of Theorem \ref{theorem_1}.
 
So  this construction shows the bound  is sharp for every $\Delta \geq1$.

For (iii),  from part (ii) above ($t \geq 1$ and $\Delta \geq 2$), we get  $t^2 + 3t +2 -2\Delta \leq 0$. Solving the quadratic and rounding up, since $t$ must be an integer, we get  \[f(G) \leq t = \left \lceil \frac{-3 + \sqrt{ 1 +8\Delta}}{2} \right \rceil,\] which holds true also for the case $\Delta=1$.
\end{proof}

\begin{theorem} \label{theorem_n}
 Let $G$ be a graph on $n \geq 4$, and $t \geq 1$  an integer such that   \[\binom{t+1}{2}+3 \leq n \leq \binom{ t + 2}{2} +2.\]  Then $f(G) \leq t$, and this is sharp  for all values of $n$ in the range. Also, for $n \geq 4$, \[ f(G) \leq \left \lceil \frac{  - 3  +\sqrt{8n-15}}{2} \right \rceil.\]
\end{theorem}

\begin{proof}
Observe that  for $\binom{t+1}{2}+3 \leq n \leq \binom{ t + 2}{2} +1$  it follows that if $|G|= n$ then   $\Delta(G)\leq n-1$ hence  $\binom{t + 1}{2}  +2 \leq \Delta(G) \leq \binom{t +2}{2}$ and $f(G) \leq t$ by Theorem \ref{theoremdelta}

We  now construct for every $n$,  such that   $\binom{t +1}{2}+3 \leq  n \leq  \binom{t +2}{2} +1$,  a graph $G_n=G$ with $f(G) = t$ proving sharpness.  
 
Let $n = \binom{t +1}{2} + j  : j =3,\ldots, t +2$.  

% ( so for t = 5  n  = 18,19,20,21,22  =  ( 5 +1 choose 2) + (5+2)  just to check  covering the range of \Delta = 17,18,19,20,21  - %Observe the case n = 17  \Delta = 16 is not considered )  ) .
 
Let   $A = \{  v_1,v_2,\ldots, v_t\}$  and $B = \{  u_1,\ldots, u_{n-t}\}$.  

Vertex $v_t$ is adjacent to all other vertices so that $\deg(v_t) = n-1 =  \binom{t +1}{2}+ j -1$.  Vertex $v_ q$, for $q = t-1,\ldots,1 $ has degree $\deg(v_q) =  \frac{q^2 +q + 2}{2} +1=a_q+1$,  where $v_q$  is adjacent to $v_t$  and to $u_1,\ldots, u_{a_q}$.   Figure \ref{graph1} shows the case when $t=3$ and $j=3$
  i.e. $n=\binom{3+1}{2} +3=9$.

\begin{figure}[h!]
\centering
\includegraphics{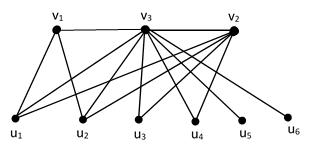} 
\caption{The graph $G_n$ for $t=3$ and $j=3$} \label{graph1}
\end{figure}

We now apply Theorem \ref{theorem_1} to $G$.  Then \[\diff(G)= \binom{t +1}{2}+j-1- \frac{(t-1)^2+t-1+2}{2}+1=t+j-3 \geq t.\]  Hence, we  apply Theorem \ref{theorem_1} by deleting $v_t$ to give a new graph $G_1$ on $n-1$ vertices  in which $\deg(v_i)$, $i=1 \ldots t-1$, as well as all the degrees of vertices in $B$ adjacent to $v_t$ are reduced by 1, and hence $\diff(G_1)=\deg(v_{t-1})-\deg(v_{t-2})=t+j-1 \geq t-1$  Therefore $f(G) \geq 1+\diff(G_1) \geq t$ and we again apply Theorem \ref{theorem_1} to delete $v_{t-1}$.  The degrees of $v_{t-2} \ldots v_1$ now remain unchanged, and for $i=2 \ldots t-1$, $\deg(v_{t-i})-\deg(v_{t-i-1})=t-i$, and the vertices are not adjacent to each other.  Hence it follows that, at each step, $\diff(G_i)=t-i$,  which, by Theorem \ref{theorem_1}, implies that $f(G)=\min\{\diff(G_j)+j: j=0,\ldots,n-2\}=t$.

Let us now look at the case when $|G| = n = \binom{t +2}{ 2}  +2$.
\begin{enumerate}
\item{If $\Delta(G)  \leq  \binom{t +2}{ 2} $. Then by Theorem \ref{theoremdelta}, $f(G) \leq t$ and we are done.}
\item{So  $\Delta = \binom{t +2}{ 2}    +1$.  Let $v_1$  and $v_2$  be such that $\deg(v_1) = \Delta$  and $v_2$ has the second largest degree.   Observe that $v_1$ is adjacent to all vertices of $G$. Now if  $\deg(v_2)  \geq  \binom{t +1}{ 2}   +2$, then $\diff(G) \leq t$ and again we are done by Theorem \ref{theorem_1}.

So $\deg(v_2) \leq  \binom{t +1}{ 2}+1    $.  We delete $v_1$ to get the graph $G_1$.  Clearly $\Delta(G_1) = \deg(v_2)  - 1  \leq \binom{t +1}{ 2} $  and by the Theorem \ref{theoremdelta}, $f(G_1)\leq t - 1$  hence $f(G) \leq t$.}
\end{enumerate}

For, sharpness we can take the graph $G$  constructed above on $n$ vertices for $ n =\binom{t +2}{ 2}  +1$ and add an isolated vertex.

Now for a graph $G$ on $n$ vertices with $4 \leq n \leq \binom{t +2}{ 2} +2$,  we know $g(G) \leq t$, hence we get  $ t^2 +3t  - 2n  + 6 \geq 0$, and solving the quadratic gives   \[f(G)  = t \leq \left   \lceil \frac{ -3  +\sqrt{ 8n -15 }}{2} \right \rceil.\]
\end{proof}

\section{Trees and Forests}

We have determined  the maximum possible value for $f(G)$  with respect to $\Delta(G)$  (Theorem \ref{theoremdelta}) and with respect to  $|G| = n$. (Theorem \ref{theorem_n}).
 
We propose the problem of finding $\max \{  f(G) : G \mbox{  is a forest  on $n$ vertices} \}$  and conjecture  the following :
 
\begin{conjecture} \label{con_trees}
If  $F$  is a forest on $n$ vertices, where   $n \leq \frac{t^3 + 6t^2 +17t +12}{6}$  then   $f(F) \leq t$  and this is sharp.
 \end{conjecture}

The following construction shows that if the conjecture is true then the upper bound is  best possible.

Consider the sequence $a_ j = \binom{ j +1}{2}  +1$.  For $t \geq 0$ we define a tree $T_t$ on $b_t$ vertices  as follows:  Let $P_{2t +3}$ be a path   on $2t +3$ vertices.   Now to the vertex  $v_{2j}$ for  $j = 1,\ldots,t+1$  of the path we add exactly  $ x_ j  = a_ j - 2  =  \binom{ j +1}{2} - 1$ leaves,  so that $ x_1 = 0$,  $x_2 = 2 $ and so on.
 
Clearly $\deg(v_1) = \deg(v_{2t+3}) = 1$,  and for $t \geq 1$,  $\deg(v_{2j +1} ) = 2$ for $j =1,\ldots,t$  while $\deg(v_{2j} ) =  a_ j$ for $j = 1,\ldots,t+1$.

Now for $t = 0$  we get $T_0=K_{1,2}$ with $b_t=3$,  for $t = 1$  we get $b_t=7$ and for $t=2$, $b_t=14$, as shown in Figure \ref{tree1}.

\begin{figure}[h!]
\centering
\includegraphics{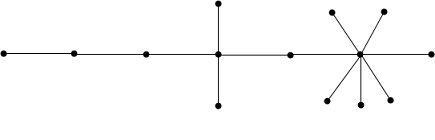} 
\caption{The tree $T_2$ on 14 vertices} \label{tree1}
\end{figure}

The number of vertices in $T_t$ is  $b_t=\frac{t^3+6t^2+17t+18}{6}$.  We prove this by induction on $t$.  For $t=0$, $b_0=3$ and for $t=1$, $b_1=7$ as required.  So let us assume it is true for $b_{k-1}$.  Then \[b_k=b_{k-1}+\frac{k^2+3k+4}{2}\] \[=\frac{(k-1)^3+6(k-1)^2+17(k-1)+18}{6}+\frac{k^2+3k+4}{2}\]\[=\frac{k^3+6k^2+17k+18}{6}\] as required.

So $\diff(T_t)= \frac{t^2+3t+4}{2} - \frac{(t-1)^2+3(t-1)+4}{2}=  t+1$.  We can, again by induction on $t$, show that $f(T_t)=t+1$, for $t \geq 1$.  Clearly  $f(T_0)=1$.  Suppose that $f(T_t) \leq t < \diff(T_t)$.  Then we should remove the vertex of degree $\Delta$ in order to obtain a minimal 2-equating set.  But this leaves isolated vertices and the tree $T_{t-1}$.  But $f(T_{t-1})=t$, by induction, and hence $f(T_t)=t+1$, a contradiction.

%We now show that Conjecture \ref{con_trees} is true for $n=13$.  We prove it for a forest of order 13, since for orders less than %13, we can add isolated vertices to complete it to a forest on 13 vertices and repeat the proof.

\begin{proposition}
Let $F$ be a forest on   13 vertices.  Then $f(F) \leq  2$.
\end{proposition}

\begin{proof}
Let $F$ be a forest on 13 vertices with degree sequence $d_1 \geq d_2 \geq d_3  \geq \ldots \geq d_{13}$, and let $u$, $v$ and $w$ be vertices of degree $d_1$, $d_2 $  and $d_3 $ respectively.
 
We observe the following facts :
 \begin{enumerate}
\item{We may assume that $d_2 \geq 4$,  for otherwise we delete $u$  and let $ F^* = F-u$. Then $\Delta(F^*) \leq 3$ and by Theorem \ref{theoremdelta},  $f(F^*) \leq 1$  hence $f(F) \leq 2$.}
\item{Therefore $d_1 - d_2 \geq 3$,  for otherwise $f(F) \leq \diff(F) \leq 2$.}
\item{Hence we may assume $d_2  \geq 4$ and $d_1 \geq  7$.}
\item{$d_3  \leq 3$ --- otherwise since  $d_1 \geq 7$, $ d_2 ,d_3 \geq 4$,  we get (even in the worst case where $u$, $v$, $w$  induce a path of three vertices in some order) $ |F| \geq 14$.}
\item{If $u$ and $v$ are not in the same component then from $|F| = 13$ we get   $F = K_{1,4} \cup  K_{1,7}$, and $f(F)  = 2$,  by deleting the centers of the stars. We now use the notation $S_{a,b}$ to denote  the double star with adjacent centres of degrees $a$ and $b$.  }
\item{If $d_2 \geq 5$ and $d_1 \geq 8$  then $|F| = 13$ if and only if  $F = S_{5,8}$  and $f(S_{5,8}) = 2$.  So we assume $d_2 = 4$  and $u$ and $v$ are in the same component.}
\item{If $u$ and $v$ are adjacent, let $F^* = F-u$. Then $\Delta(F^*) \leq 3$ and by Theorem \ref{theoremdelta} $f(F^*) \leq 1$ hence $f(F) \leq 2$.}
 \end{enumerate}

Therefore,  $F$ can be one of the following  graphs:
 \begin{itemize}
\item{$S_{4,8}$ with the edge between the centres subdivided.  Clearly $f(F) = 2$.}
\item{$S_{4,7}$  with the edge between the centres subdivided twice.  Clearly $f(F) = 2$.}
\item{$S_{4,7}$ with the edge between the centres subdivided and another vertex added  adjacent to a leaf of the vertex of degree 7 or of degree 4.  In both two cases $f(F) = 2$.}
\item{$S_{4,7}$ with the edge between the centres subdivided by a vertex $w$ and to the vertex  $w$ we attach a leaf so that $\deg( w)  = 3$. Clearly $f(F) = 2$.}
\end{itemize}
 
In all these cases we only need to delete vertices $u$ and $v$ to get at least two vertices of maximum degree, and hence $f(F) \leq 2$.
 
Observe that if $|F| < 13$ we may add $0$-vertices to get a forest on 13 vertices and the same argument applies, hence for $|F|  \leq 13$,  $f(F)  \leq 2$..
%Note that if $|F|<13$, one can add isolated vertices  to increase the order to 13, and the above result applies.

\end{proof}
 
For $n=14$ we have the graph $T_2$ (Figure \ref{tree1}) which has exactly $\frac{2^2+6(2^2)+17(2)+18}{2}=14$ vertices and we know that $f(T_2)=3$.

We now prove the following result:
\begin{theorem} \label{forests}
Let $F$ be a forest on $n$ vertices and $k \geq 2$ an integer.  Suppose $n^{\frac{1}{3}} \geq 2k-1$.   Then \[f_k(F) \leq (2k-1) \lfloor n^{\frac{1}{3}} \rfloor.\]
\end{theorem}

We first prove the following lemmas:

\begin{lemma} \label{lemma_tree_1}
For every $k\geq 2$ and every graph $G$ with maximum degree $\Delta$, $f_k(G) \leq (k-1)\Delta$.
\end{lemma}

\begin{proof}

By induction on $\Delta$.  If $\Delta=0$, then the result is trivial since either $|G| <k$ or there are $k$ vertices of degree 0.

So suppose the result holds for $\Delta=r$ and let $G$ have $\Delta=r+1$.  If there are $k$ vertices of maximum degree $r+1$ we are done.  Otherwise, remove all the vertices of maximum degree $\Delta$  in $G$ ---there are at most $k-1$ such vertices.  The resulting graph $H$ has maximum degree $r$ and hence by the induction hypothesis, $f_k(H) \leq (k-1)r$.  Hence \[f_k(G) \leq (k-1)r+k-1 = (k-1)(r+1)=(k-1)\Delta(G)\] as required.
\end{proof}

\begin{lemma} \label{lemma_tree_2}
Let $G$ be a forest and let $A$ be any subset of $k \geq 2$ vertices of $G$.  Define $M(A)$ to be the set of vertices of $V(G)\backslash A$ each having at least two neighbours in $A$.  Then $|M(A)|<|A|=k$.
\end{lemma}

\begin{proof}
Suppose $|M(A)| \geq |A|=k$.  Let $B$ be any subset of $M(A)$ of cardinality $k$ and let $H$ be the bipartite graph with vertices $A \cup B$ and only those edges connecting vertices in $A$ to vertices in $B$.  Since each vertex of $B$ has degree at least 2 in $H$, $|E(H)| \geq 2k$.  But $|V(H)|=2k$, therefore $H$ has a cycle contradicting the fact that $G$ is a forest.
\end{proof}

We now prove Theorem \ref{forests}.

\begin{proof}

The degree of $F$ can range from $0$ to $n-1$.  Let us divide this range into subintervals \[S_j=[jn^{\frac{1}{3}}, (j+1)n^{\frac{1}{3}} ) \mbox{ for } j=0,\ldots,\lfloor n^{\frac{1}{3}} \rfloor-1\] with the last two intervals being $S_{\lfloor n^{\frac{1}{3}} \rfloor}=[ \lfloor n^{\frac{1}{3}} \rfloor n^{\frac{1}{3}}, n^{\frac{2}{3}})$ and $S_{L}=[\lceil n^{\frac{2}{3}} \rceil ,n)$.  Let us denote by $A_j$ or $A_L$ the set of vertices of $F$ whose degrees fall in the intervals $S_j$ or $S_L$ respectively.

%  Note that for any vertex in $A_j$, $j < n^{\frac{1}{3}}$,  its degree is bound by $jn^{\frac{1}{3}}$ and $(j+1)n^{\frac{1}{3}}$ %%($n^{\frac{2}{3}}$ and $n$).

We first claim that $A_L$ contains at most $\lfloor n^{\frac{1}{3}} \rfloor$ vertices.  Suppose not and let $x_j$ be the number of vertices of $F$ having degree $j$.  Consider the forest $F^*$ on $n^* \leq n$ vertices obtained by deleting all isolated vertices in $F$.  Then clearly the number of vertices $x_j$ of degree $j \geq 1$ in $F^*$ in the same as in $F$, and we have $x_1+\ldots+x_{n-1}=n^*$ and $x_1+2x_2+\ldots+(n-1)x_{n-1} \leq 2n^*-2$.  Multiplying the first equation by 2 and subtracting the second gives: \[x_1-\sum_{j=3}^{n-1}(j-2)x_j \geq 2.\]  Hence \[x_1 \geq 2 + \sum_{j \geq 3}(j-2)x_j.\]

In particular \[x_1 \geq 2+(n^{\frac{2}{3}}-2)|A_L\]\[\geq 2(n^{\frac{2}{3}}-2)(n^{\frac{1}{3}}+1)=n+n^{\frac{2}{3}}-2n^{\frac{1}{3}}.\]

But \[n^* \geq x_1+|A_L| \geq n+n^{\frac{2}{3}}-2n^{\frac{1}{3}}+n^{\frac{1}{3}}+1=n+n^{\frac{2}{3}}-n^{\frac{1}{3}}+1>n^*\]
a contradiction.

We now proceed as follows.  We remove from $F$ the vertices in $A_L$ and redistribute the resulting degrees among the intervals $S_j$, $j=0,\ldots,\lfloor n^{ \frac{1}{3}} \rfloor$, recalculating  $A_ j$ for $j=0,\ldots,\lfloor n^{\frac{1}{3}} \rfloor$.  If there are at least $k$ vertices with degrees in the last interval we stop.  Otherwise we remove these vertices and again we redistribute the degrees among the intervals $S_j$, $j=0,\ldots, \lfloor n^{\frac{1}{3}} \rfloor-1$, recalculating  $A_ j$ for $j=0,\ldots,\lfloor n^{\frac{1}{3}} \rfloor-1$.  This process continues until we reach one of the following possibilites:
\begin{enumerate}
\item{We have deleted all vertices and we are left with only those vertices in $A_0$;}
\item{For some $j \geq 1$, $A_j$ contains at least $k$ vertices.}
\end{enumerate}

We consider these cases separately:
\begin{enumerate}
\item{  In this case we have deleted $\lfloor n^{\frac{1}{3}} \rfloor $ vertices from $A_L$ and at most $(k-1)\lfloor n^{\frac{1}{3}} \rfloor $ further vertices by deleting at most $(k-1)$ vertices that were in the respective sets $A_1,\ldots,A_{\lfloor n^{\frac{1}{3}} \rfloor }$ at each stage.  So altogether $k\lfloor n^{\frac{1}{3}} \rfloor $ vertices have been deleted.

But now the resulting graph has maximum degree at most $\lfloor n^{\frac{1}{3}} \rfloor $ and therefore, by Lemma \ref{lemma_tree_1}, by deleting at most a further $(k-1)\lfloor n^{\frac{1}{3}} \rfloor $ vertices we arrive at a graph with $k$ vertices of maximum degree (or at most $k-1$ vertices at all).  To do this we have altogether deleted at most $(2k-1)\lfloor n^{\frac{1}{3}} \rfloor $ vertices, as required.}
\item{We have stopped the deletion process when $A_j$, $j \geq 1$, contains at least $k$ vertices, $A_j$ being the set of vertices of the reduced forest having degrees in $S_j=[jn^{\frac{1}{3}},(j+1)n^{\frac{1}{3}})$.

Let $v_1,v_2,\ldots,v_k$ be the $k$ vertices in $A_j$ of largest degrees, say $d_1 \geq d_2\geq \ldots \geq d_k$.  Let us call this set of vertices $A$.   By Lemma \ref{lemma_tree_2}, $|M(A)|<k$ where we recall that $M(A)$ is the set of vertices adjacent to at least two vertices of $A$.  Since a vertex $v \in A$ can be adjacent to at most $k-1$ other vertices in $A$ and $k$ vertices in $M(A)$, there are at least $\deg(v)-2k+1$ vertices that are neighbours of $v$ but which are not in $A \cup M(A)$.  Since such vertices are adjacent to at most one vertex from $A$, these $\deg(v)-2k+1$ vertices are only adjacent to $v \in A$ and not to any other vertex in $A$.  Let $B(v)$ be the set of these neighbours of $v$.

Now consider any vertex $v_i \in A$, $i=1 \ldots k$.  Suppose $\deg(v_i)=\deg(v_k)+t_i$.  Then \[|B(v_i)| \geq \deg(v_i) - 2k+1 \geq \deg(v_k)+t_i-2k+1 \geq n^{\frac{1}{3}}+t_i-2k+1 \geq t_i\] since $n^{\frac{1}{3}} \geq 2k-1$.

We therefore need to remove $t_i$ vertices of $B(v_i)$ (and this will not change the degree of any other vertex in $A$) in order to equate $\deg(v_i)$ and $\deg(v_k)$.  However, since $|B(v_i)| \geq t_i$, we can do this.

Hence, equating all the degrees of the vertices $v_1,\ldots,v_{k-1}$ to $\deg(v_k)$ can be done at the cost of  deleting  at most a further $(k-1)\lfloor n^{\frac{1}{3}} \rfloor $ vertices.  This means that we have deleted altogether at most $(2k-1)\lfloor n^{\frac{1}{3}} \rfloor $ vertices, so we are done.}
\end{enumerate}
\end{proof}

Remark:
The above proof also works for the more general class of graphs without even cycles.  Lemma \ref{lemma_tree_2} remains unchanged since the graph $H$ used in the proof is bipartite by construction.  A graph $G$ on $n$ vertices and without even cycles contains at most $\frac{3n}{2}$ edges \cite{bollobas2004extremal}.  Therefore, in the proof of the Theorem, instead of the computation involving $x_1$ we compute an upperbound on the number of vertices in $A_J$, by noting that if this number is at least $3n^{\frac{1}{3}}+1$ then \[3n \geq 2|E(G)| \geq \sum_{v \in A_J} \geq (3n^{\frac{1}{3}}+1)(n^{\frac{2}{3}}) = 3n + n^{\frac{2}{3}}>3n,\] a contradiction.

We then remove the vertices of $A_J$ and redistribute the resulting degrees, sacrificing at most $3n^{\frac{1}{3}}$ vertices, and continue as in the proof.  This gives \[f_k(G) \leq (2k+1)n^{\frac{1}{3}},\] giving a weaker bound for a more general class of graphs.

\section{ The functions $g(\Delta,k)$ and $h(\Delta,k)$}
 
Lemma \ref{lemma_tree_1}, which states that $g(\Delta,k) \leq ( k-1)\Delta$, plays a crucial rule in the proof  of Theorem \ref{forests}.  Another motivation to study $g(\Delta,k) =  \max\{ f_k(G): \Delta(G)  \leq \Delta \}$ comes from the Proposition \ref{prop1} below,  which gives a weak support for the conjecture $f_k(G) =< f(k) \sqrt{|G|}$ mentioned in the introduction,  and also demonstrates that for graphs with $e(G) = o(n^2)$,  $f_k(G) = o(n)$  (where $e(G)$  is the number of edges of $G$) .   
 
Observe that it has not yet been proved in general  that for fixed $k$ and $G$ a graph on $n$ vertices,  $f_k(G) = o(n)$.
 
\begin{proposition} \label{prop1}
 
Suppose $G$ is a graph on $n$ vertices and $e(G) \leq cn^{1+\beta}$ where $0 \leq \beta < 1$ and let $\alpha= \frac{1+ \beta}{2}$.  
Then $f_k(G)  \leq  (k - 1 +2c)n^{\alpha}$,  and in particular for $\beta= 0$,  $f_k(G) \leq (k-1+2c) \sqrt{n}$.
 \end{proposition}
\begin{proof}
Define $V_{\alpha} = \{  v:  \mbox{  }\deg(v)  \geq  n^{\alpha} \}$ and suppose $|V_{\alpha})|  > 2cn^\alpha$.
 
Then \[2e(G)  = \sum \{\deg(v) : v \in  V(G)\} \geq  \sum \{ \deg(v): v \in V_{\alpha} \} > n^{\alpha}2cn^{\alpha} = 2cn^{2\alpha} =  2cn^{1+\beta} \geq 2e(G),\]  a contradiction.
 
Hence $|V_{\alpha}| \leq 2cn^{\alpha}$.  Delete $V_{\alpha}$ we get a graph $H$ with $\Delta(H) \leq n^{\alpha}$.  Hence applying Lemma \ref{lemma_tree_1} we get \[f_k(G) \leq |V_{\alpha}| +(k-1)\Delta(H) \leq 2cn^{\alpha} + (k-1)n^{\alpha}  = (k -1 +2c) n^{\alpha}.\]
 
\end{proof}

So a better knowledge of the behavior of $g(\Delta,k)$ will  help to obtain better bound on $f_k(G)$ as well as $f(n,k) =  \max \{ f_k(G) : |G|= n \}$.       
 
\begin{proposition} \label{gDelta}
For every $\Delta \geq 0$ and $k \geq 2$,  
 \begin{enumerate}
\item{$g(0,k) = 0$.}
\item{$g(1,k) = \lfloor\frac{k-1}{2 } \rfloor$.}
\item{For $\Delta \geq 1$,  $g(\Delta,2)  =  \left \lceil \frac{-3+\sqrt{8\Delta+1}}{2} \right \rceil$.}
\end{enumerate}
\end{proposition}

\begin{proof}
\begin{enumerate}
\item{Clearly if $G$ is a graph with maximum degree $\Delta =  0$  then either $|G| \geq k$ and we are done, or else $|G| \leq k-1$ and we are done by the definition of $f_k(G)$,  hence $g(0,k) = 0$.}
\item{Consider $G$  with maximum degree $\Delta(G) = 1$.  If there are already $k$ vertices of degree 1 we are done.  So assume there are at most $k-1$ vertices of degree 1. By  parity these vertices form exactly $ \lfloor \frac{ k-1}{2 } \rfloor$  isolated edges containing exactly $2 \lfloor\frac{k-1}{2 } \rfloor$ vertices of degree 1.  We delete from each isolated edge one vertex of degree 1  to get an induced subgraph $H$ with $\Delta(H) = 0$.  It follows, since $g(0,k)=0$, that $f_k(G) \leq  \lfloor \frac{ k-1}{2 } \rfloor$. This bound is sharp as demonstrated by the graph $tK_2  \cup mK_1$  where $m \geq 0$  and $t =\lfloor \frac{ k-1}{2 } \rfloor$, $ k \geq 3$.   }
\item{This is a restatement of Theorem \ref{theoremdelta}.}
\end{enumerate}
 \end{proof}

Determining $g(2,k)$  requires more efforts, in particular we will use Ore's observation that if $G$ is a graph on $n$ vertices  without isolated vertices, then the domination number of $G$, denoted $\gamma(G)$,   satisfies  $\gamma(G) \leq \lfloor \frac{n}{2 } \rfloor$ \cite{ore1962theory}.
 
\begin{theorem}
 
For $k \geq 2$,  $g(2,k) = k-1$.
 \end{theorem}

\begin{proof} 
The graph $G =  (k-1) K_{1,2}$  ($ k-1$  vertex-disjoint  copies of the star $K_{1,2}$)  has $f_k(G) = k-1$ as is easily checked.  So $g(2,k)  \geq k-1$.
 
Let us prove the converse.
 
Consider a graph $G$ with $\Delta(G) = 2$,  otherwise by Proposition  \ref{gDelta}  (part 2) we are done.
 
Let $ n_2  =  |\{  v : \deg(v) =  2\} |$.  Clearly if $n_2  \geq k$ we are done so we may assume $1 \leq n_2 \leq k-1$.
 
We collect  the (possible)  components of $G$ into three  subgraphs:  $A =  \{$ all isolated vertices  and isolated edges$\}$, $B = \{$all copies of $K_{1,2 }\}$,  $C  = \{$all other components$\}$.
 
We denote by $t$ the number of copies of $K_{1,2}$  in $B$  and we observe that  $t \leq n_2$ and that in each component in $C$ the vertices of degree 2 induce either a path  (including a single edge)  or a cycle.
 
We claim that if $t > \lfloor \frac{k-1}{2 } \rfloor$ we are done by deleting all $n_2-t$ vertices of degree 2 in $C$ and from each copy   of $K_{1,2}$ in $B$  we delete a leaf  to get  from $G$ an induced subgraph $H$ with $\Delta(H) = 1$ and with at least $2 (\lfloor \frac{k-1}{2 } \rfloor + 1)  \geq  k$ vertices of degree 1,  and we have deleted altogether $n_2 =< k-1$  vertices.   So we shall assume  $t \leq  \lfloor \frac{k-1}{2 } \rfloor$.
 
Consider the subgraph $F$  induced by the vertices of degree 2 in $C$. 
 
\medskip 

\noindent Case  1:  $|F| = 0$.
 
If  $|F| = 0$  (namely $C$ is empty)  then $n_2 = t  \leq \lfloor\frac{k-1}{2 } \rfloor$.   Delete a leaf from each copy of $K_{1,2}$ in $B$.  We get from  $G$ a graph $H$ with  $\Delta(H) = 1$ (as in $A$ all components have maximum degree at most 1).
 
If in $H$ there are already $k$ vertices of degree 1, we are done as we have deleted  $t \leq   \lfloor \frac{k-1}{2 }  \rfloor$ vertices.     Otherwise by Proposition  \ref{gDelta} (part 2), $ f_k(H) \leq \lfloor \frac{k-1}{2 } \rfloor$  and hence $f_k(G)  \leq   2 \lfloor \frac{k-1}{2 } \rfloor \leq k-1$.    
  \medskip

\noindent Case 2:  $|F| > 0$.

Then as we have noted before, due to the components of $C$,  there are no isolated vertices in $F$,  and by  Ore's result $\gamma(F) \leq \lfloor\frac{n_2 - t }{2 } \rfloor \leq  \lfloor \frac{k- 1 - t}{2 } \rfloor$.
 
Let $D$ be a dominating set for $F$ that realises $\gamma(F)$, hence $|D| \leq \frac{n_2 - t}{2}$.
 
Delete $D$ and consider the induced subgraph $H$ on  $A \cup C$.  Clearly $\Delta(H) \leq 1$ and denote by $n_1$ the number of vertices of degree 1 in $H$.
 
Now we look again at $B$.  

\medskip
 
\noindent Case 1: $ t  = 0$.
 
Since $t = 0$,  $B$ is empty, and either $n_1  \geq k$ and we are done as we have deleted $|D| =\leq \lfloor \frac{n_2}{2} \rfloor \leq  \lfloor \frac{k-1}{2 } \rfloor$ vertices  or  by Proposition  \ref{gDelta} (part 2),  $f_k(G) \leq f_k( H)+ |D| \leq  2 \lfloor\frac{k-1}{2 } \rfloor \leq k-1$.
 
\medskip

\noindent Case  2:  $1 \leq t \leq \lfloor \frac{k-1}{2 }  \rfloor$.
 
\begin{enumerate}
\item{if $n_1 \geq  k - 2t$ then deleting a leaf  from every copy of $K_{1,2}$  in $B$,  we get  an induced graph  $H^*$ on  $A \cup B \cup C$  (extending $H$ to  the leftover of $B$) with $\Delta(H^*) = 1$ and at least $k- 2t +2t = k$ vertices of degree 1 and we are done as we have deleted altogether \[|D| + t  \leq \frac{n_2 - t }{2 }+t  =  \frac{n_2 +t}{2} \leq n_2 \leq k-1\] vertices.}
\item{If $n_1 \leq k - 1 -2t$  (recall  $n_1$ is the number of vertices of degree 1 in $H$ formed from $A  \cup  \{C \backslash D\} $), then we delete $\frac{n_1}{2}$ independent  vertices of degree 1 in  $H$,  and $t$ vertices of degree 2 in $B$ to get an induced subgraph $H^*$ with $\Delta(H^*) = 0$.
 
But $g(0,k) = 0$  hence $f_k(H^*) = 0$  and \[f_k(G)  \leq  \frac{n_1}{2} + |D| + t \leq  \frac{k-1-2t}{2} + \frac{k-1-t}{2} + t = \frac{2k - 2 - t}{2} \leq k-1\] and the proof is complete.}
\end{enumerate}
 
\end{proof}

The following construction supplies a lower bound for $g(\Delta,k)$ in terms of $g(\Delta,2)$.

\begin{proposition}
For even $k \geq 2$, $g(\Delta,k) \geq g(\Delta,2) \frac{k}{2} + \frac{k}{2}-1$.

\end{proposition}

\begin{proof}

Recall the sequence $a_t = \binom{t + 1}{2}  +1$,  which for $t \geq 1$ gives the smallest maximum degree  for which there is a graph $G$ with $f(G)  =  t$.   Such a graph is $\bigcup K_{1,a_j }$ for $j = 1,\ldots,t$  and in  case we have $a_t \leq \Delta < a_{t+1}$,  $G = K_{1, \Delta} \cup K_{1,a_ j}: j = 1,\ldots,t-1$.
 
Now we take  $k-1$ copies of  $K_{1,a_t}$   and $\frac{k}{2}$ copies of $K_{1,a_j}: j = 1,\ldots,t-1$.  In case $a_t \leq \Delta < a_{t+1}$ we take $k-1$ copies of $K_{1,\Delta}$ and $\frac{k}{2}$  copies of $K_{1,a_j}:  j = 1,\ldots, t-1$.  
 
Note that for $k = 2$ this is exactly the sequence that realises Theorem \ref{theoremdelta}.
 
Observe now that we cannot equate  to degree $\Delta$ as there are just $k-1$ such degrees.  So we can equate to the second largest degree $a_{t-1}$ by deleting exactly $\Delta - a_{t-1}$  leaves from $\frac{k}{2}$ vertices of the maximum degree and $k -1 - \frac{k}{2}$ other centres.  Altogether we deleted
 
\[\frac{(\Delta - a_t)k}{2}  + \frac{k}{2}-1 \geq \frac{g(\Delta,2)k}{2}  + \frac{k}{2}  - 1.\]   In case $\Delta = a_t$ we have deleted exactly $\frac{g(\Delta,2)k}{2}  + \frac{k}{2}-1$.
 
We can now equate to some value $x$ such that  $a_{t-1}>   x  \geq  a_{t-2}+j$, $ j \geq 1$.    However clearly this requires the deletion of more vertices then just to equate to $a_{t-1}$  and in particular the deletion of at least    $g(\Delta,2) \frac{k}{2} + \frac{k}{2}- 1$  vertices. 

Now we can try to equate to $a_{t-2}$.
 
The cheapest way is to delete the $k-1$ vertices of degree $\Delta$ and $a_{t-1} -a_{t-2}$  leaves from each of the $\frac{k}{2}$ vertices of degree $a_{t-1}$.  So altogether we deleted\[k - 1 + (a_{t-1}- a_{t-2})\frac{k}{2}=   k-1 +  (g(\Delta,2) - 1)\frac{k}{2}  = g(\Delta,2)\frac{k}{2}  + \frac{k}{2} - 1\] vertices.

Again we can now try to equate to some value $x$ such that  $a_{t-2}>   x  \geq  a_{t-3}+j$, $ j \geq 1$..  However Clearly this requires the deletion of more vertices then just to equate to $a_{t-2}$ and in particular the deletion of at least    $g(\Delta,2) \frac{k}{2} + \frac{k}{2}- 1$  vertices. 

So this deletion process continues and we always forced to delete at least  $g(\Delta,2) \frac{k}{2} + \frac{k}{2}- 1$  vertices, even if we delete all the centres of the stars to get an induced subgraph with all degrees equal 0.

%The rest gives always a deletion of $g(\Delta,2)\frac{k}{2} + \frac{k}{2}  -1$ vertices   and at least this number even if we drop all %the centres to get only the original leaves as vertices of degree 0.
 
Hence for even $k \geq 2$  we get $g(\Delta,k) \geq g(\Delta,2)\frac{k}{2} + \frac{k}{2}  -1$ (which is sharp for $k=2$).
 
While slight improvements on this lower bound are possible for  odd $k \geq 3$,  our goal in this construction is only to demonstrate a linear lower bound  on $g(\Delta,k)$  in terms of $g(\Delta,2)$ and $k$ for which the construction suffices.

\end{proof}
We now turn our attention to $h(\Delta,k)$.  Recall that  for $ k \geq 2$, a graph $G$  is $k$-feasible if it contains an induced subgraph $H$ (possibly also $H = G$) such that in $H$ there are at least $k$ vertices that realise  $\Delta(H)$, and we define   \[h(\Delta,k)  = \max \{ |G| : \Delta(G) \leq \Delta \mbox{ and $G$  is not $k$-feasible}\}.\]  
 
\begin{theorem}
For every $\Delta \geq 0$ and $k \geq 2$,  
\begin{enumerate}
\item{$h(\Delta,k)  \leq R(k,k)  - 1$.}
 \item{$ h(0,k) = k-1$.}
\item{$h(1,k) =  \lfloor \frac{k}{2} \rfloor + 2 \lfloor \frac{k-1}{2} \rfloor$.}
\item{For odd $k \geq 3$, $h(2,k) = 2k-2$,   and for even $k \geq 2$,  $h(2,k) = 2k-3$.}
\item{$h(\Delta,k)  \leq g(\Delta,k)  + k-1$.}
\end{enumerate}
\end{theorem}
 \begin{proof}
\begin{enumerate}
\item{Clearly if $|G| \geq R(k,k)$  then $G$ has  a vertex-set $A$,  $|A| \geq k$  such that the induced subgraph on $A$  is  either a clique or an independent set .
 
Hence deleting $V - A$  we are left with a regular graph on at least $k$ vertices hence $G$ is $k$-feasible and $h(\Delta,k) \leq R(k,k)- 1$.}
\item{$h(0,k) = k-1$ is trivially realised by $( k-1) K_1$ i.e. $k-1$ isolated vertices.}
\item{A lower bound for $h(1,k)$  is   $h(1,k)  \geq \lfloor\frac{k}{2} \rfloor + 2 \lfloor \frac{k-1}{2} \rfloor$  realised by the graph $G =  \lfloor \frac{k}{2} \rfloor K_1  \cup \lfloor \frac{k-1}{2} \rfloor K_2$  which is trivially seen to be non-$k$-feasible.
 
Next suppose $G$ is a graph having  $\lfloor\frac{k}{2} \rfloor + 2 \lfloor \frac{k-1}{2} \rfloor  + j$  vertices, $j \geq 1$.
 
Write  $\lfloor\frac{k}{2} \rfloor + 2 \lfloor \frac{k-1}{2} \rfloor  + j  =   x  +2y$   where $x$ denotes the number of $0$-vertices  and $2$y the number of $1$-vertices in $G$.   
 
Now if $y >  \lfloor \frac{k-1}{2} \rfloor$ we have at least $k$ vertices of degree 1 and we are done.  If  $0 \leq y \leq  \lfloor \frac{k-1}{2} \rfloor$ then delete $y$ $1$-vertices, one of each  copy of $K_2$,  and we are left with at least  $\lfloor\frac{k}{2} \rfloor +  \lfloor \frac{k-1}{2} \rfloor +  j \geq  \lfloor\frac{k}{2} \rfloor + \lfloor \frac{k-1}{2} \rfloor  +  1  = k$ vertices of degree  0.
 
Hence $G$ is $k$-feasible, and  $h(1,k) =  \lfloor\frac{k}{2} \rfloor + 2 \lfloor \frac{k-1}{2} \rfloor$.}
\item{Clearly $h(2,k)  \leq 2k - 2$  since if  $G$ has $\Delta = 2$  and at least $2k- 1$ vertices  then by deleting at most $k- 1 =  g(2,k)$ vertices we cannot get below $k$ so there must be induced $H$  with at least $k$  vertices realizing the maximum degree. 
 
Suppose $k$  \textbf{is odd}  and $k \geq 3$.  Consider the graph $G =   \frac{k-1}{2}P_4$  ($\frac{k-1}{2}$ copies of the path on four vertices $P_ 4$).  Clearly $|G| = 2k-2$ having exactly $k-1$ $2$-vertices and $k-1$  $1$-vertices.
 
Observe that  if $G$ is $k$-feasible then in at least one of the $P_4$  we should be able to delete just one vertex to get the remaining three vertices of the same degree, otherwise if in each  copy of $P_4$ (or what remains of it after deleting some vertices) we will  have at most two vertices of the same degree then over all $G$ we will have at most $k-1$ vertices of the same degree,  meaning $G$  is not $k$-feasible.
 
However it is impossible to delete one vertex from $P_4$ to get all the remaining three vertices of the same degree hence $G$ is not $k$-feasible proving $h(2,k) = 2k-2$ for odd $k \geq 3$.

Suppose $k$ \textbf{is even},  $k \geq 2$.
 
The case $k = 2$ is trivial hence we assume$k \geq 4$.
 
Consider the graph $G  =  \frac{k-1}{2}P_4  \cup  K_1$.  Clearly $|G| = 4\frac{k-2}{2} +1 = 2k-3$ having exactly  $k-2$  $1$-vertices,  $k-2$  $2$-vertices and one  $0$-vertex.   If $G$ was $k$-feasible then by deleting the $0$-vertex $v$, $H = G-v$  would be at least  $k-1$-feasible with odd  $t = k -1 \geq 3$.
 
But $H$ is exactly the graph which was proved above to be non $t$-feasible for odd $t \geq 3$,  so  $G$ is not $k$-feasible, proving $h(2,k)  \geq 2k-3$ for even $k \geq 4$.
 
We have to  show that if $|G| = 2k-2$ and $\Delta(G)=2$,  then for even $k \geq 4$, $G$ is $k$-feasible,  this will complete the proof that for even $k \geq 2$,  $h(2,k) = 2k-3$.

Suppose on the contrary that  $|G| = 2k-2$  and $\Delta(G) = 2$ but $G$ is not $k$-feasible.  Let $n_j$, $j = 0,1,2$  be the number of vertices of degree $j = 0,1,2$  respectively in $G$.  
 
Since $G$ is non-$k$-feasible  and by the value of $h(1,k)$  we may assume $1 \leq n_2  \leq  k-1$.   However  $2k  - 2 > h(2,k-1) = 2(k-1) -2 = 2k-4$. Hence $G$ is $k-1$-feasible.
 
So either $n_2 = k-1$  or else, by removing  at most  $k -2$  vertices, we get  an induced subgraph $H$,  $|H| > = k$ with at least $k-1$ vertices realising the maximum degree of $H$.
 
If $\Delta( H) = 1$  then, since $k - 1$ is odd, it forces that there are  at least $k$ $1$-vertices  but then $G$ is $k$-feasible.  Otherwise $\Delta(H) = 0$ but $|H| \geq k$ and again $G$ is $k$-feasible. 
 
 So   only the case $n_2 = k-1$  is left.  Since $n=2k-2$ and $n_2=k-1$ then by parity $n_1 \leq k-2$  and $ n_0 \geq 1$.
 
We collect  the (possible)  components of $G$  into three  subgraphs:  $A = \{$all isolated vertices  and isolated edges$\}$,  $B = \{$all copies of $K_{1,2}\}$,  $C  =\{$all other components$\}$.
 
We denote by $t$ the number of copies of $K_{1,2}$  in $B$  and also observe that   $t < n_2 = k-1$ since otherwise $|G| = 3k-3 > 2k-2=|G|$  a contradiction since $k \geq 2$.  
 
Also observe  that in each component in $C$ the vertices of degree 2 induced either on  a path (including a single edge)  or a cycle.
 
Claim: If $t > \lfloor \frac{k-1}{2} \rfloor$ we are done.
 
This is because $F$ is not empty since $t < n_2$,  and by the observation above $\delta(F) \geq 1$  hence by Ore's result  the domination number of $F$ satisfies $\gamma(F) \leq \lfloor \frac{n_2 - t }{2 } \rfloor \leq  \lfloor \frac{k-1-t }{2 } \rfloor$.
 
Let $D$ be a minimum dominating set for $F$. Deleting $D$  from $C$  and from each copy  of $K_{1,2}$  in $B$ we delete a leaf  to get  an induced subgraph $H$ with $\Delta(H) = 1$ and with at least $2 (\lfloor \frac{k-1}{2 } \rfloor +1) \geq k$ vertices of degree 1, meaning $G$ is $k$-feasible.  Observe we have deleted at most \[t + \lfloor \frac{n_2 - t }{2 } \rfloor \leq \frac{n_2+t }{2 } \leq \lfloor \frac{2n_2 -1 }{2 } \rfloor = \lfloor \frac{2k-3}{2} \rfloor = k-2\] vertices, proving the claim.

Consider the subgraph $F$ induced by the vertices of degree 2 in $C$ and recall $ |F| > 0$,  hence $|F| \geq 2$.   
 
Then as we have noted before, due to the components of $C$,  there is no isolated vertices in $F$,  and by  Ore's result $\gamma(F) \leq \lfloor \frac{n_2 - t }{2 } \rfloor \leq  \lfloor \frac{k -1 - t}{2 } \rfloor$.
 
Let $D$ be a dominating set for $D$ that realises $\gamma(F)$, hence $|D|\leq \frac{ n_2 - t}{2}$.
 
Delete $D$ and consider the induced subgraph $H$ on  $A \cup C$.   Clearly $\Delta(H) \leq 1$ and denote by $x(1)$ the number of vertices of degree 1 in $H$.
 
Now let us look again at $B$.  
 
\noindent Case  1:  $t  = 0$.
 
Since $t = 0$, $B$ is empty, and we have deleted $|D| \leq \lfloor \frac{n_2}{2} \rfloor \leq  \lfloor \frac{ k-1}{2 } \rfloor = \frac{k-2}{2}$ vertices since $k$ is even.  So the number of vertices remains is  at least $2k-2  - \frac{k-2}{2}  =  \frac{3k-2}{2}$.
 
But as $k$ is even, $ h(1,k)  =   \lfloor \frac{k}{2 } \rfloor + 2\lfloor \frac{k-1}{2} \rfloor= \frac{ k}{2} + \frac{2(k-2)}{2}  =  \frac{3k - 4}{2}  < \frac{3k-2}{2}$ hence $H$ is $k$-feasible  and so $G$ is $k$-feasible.

\noindent Case  2:  $1 \leq  t \leq \lfloor\frac{k-1}{2 } \rfloor$.

We consider two cases:
\begin{enumerate}
\item{if $x(1) \geq  k - 2t$  then deleting a leaf  from every copy of $K_{1,2}$ in $B$  we get  an induced graph  $H^*$ on  $A \cup B \cup C$  (extending $H$ to  the leftover of $B$) with $\Delta(H^*) = 1$ and at least $k- 2t +2t = k$ vertices of degree 1 and we are done as we have deleted altogether $|D| + t \leq \frac{n_2 - t }{2} +t  = \frac{n_2 +t}{2}  \leq k-2$  (as before).  Hence $G$ is $k$-feasible.  }
\item{if $x(1) \leq k - 1 -2t$  (recall  $x(1)$  is the number of vertices of degree 1 in $H$ formed from $A  \cup  \{ C \backslash  D \}$), then by the even parity of $x(1)$ and as k is even we must have $x(1) \leq k-2-2t$.
 
Now  delete $\frac{x(1)}{2}$ independent  vertices of degree 1 in  $H$,  and $t$ vertices of degree 2 in $B$ to get an induced subgraph $H^*$  with $\Delta(H^*) = 0$.
 
We have removed  \[\frac{x(1)}{2} + |D| + t \leq  \frac{k-2-2t}{2} + \frac{k-1-t}{2} + t = \frac{2k - 3 -t }{2} \leq \frac{2k-4}{2 }= k-2\] vertices (since $t \geq 1$),  hence $|H^*| \geq k$ and we have $k$ vertices of degree 0  realizing $\Delta(H^*)$. Hence $H^*$ is $k$-feasible and so does $G$,  completing the proof.}
 
 \end{enumerate}}
 
\item{ Suppose $|G|= g(\Delta,k) +k$ and  $\Delta(G) = \Delta$.  Then by the definition of $g(\Delta,k)$, by deleting at most $g(\Delta,k)$ vertices we either get below $k$ vertices or have an induced subgraph $H$ with at least $k$ vertices realizing the maximum degree of $H$.
 
But deleting $g(\Delta,k)$  vertices from $G$ will leave us with a graph on at least $k$  vertices hence  the second possibility above holds and $G$ is $k$-feasible, and we conclude that $h(\Delta,k)  \leq  g(\Delta,k) + k-1$.}
\end{enumerate}
\end{proof}

\section{Open Problems}

We conclude by  proposing the following open problems:
\begin{enumerate}
\item{Certainly the most intriguing problem is to solve the Caro-Yuster conjecture that $f(n,k) \leq f(k) \sqrt{n}$.  As mentioned we proved that $ f(2) = \sqrt{2}$  is sharp and best possible, and it is known that $f(3)  \leq 43$. For $k \geq 4$ the conjecture remains open.  Even a proof that $f(n,k) =o(n)$ is of interest.}
\item{Theorem \ref{theorem_1} supplies an $O(n^2)$  algorithm to compute $f(G)$. Can $f_3(G)
$  be computed in polynomial time?}
\item{We have calculated, in section 4  , the exact values of  $g(\Delta,k)$ for $\Delta= 0,1,2$, and we have given a general  constructive lower bound for $g(\Delta,k)$.  Determining $g(3,k)$  seems a considerably more involved task, as well as proving a conjecture inspired by the Caro-Yuster conjecture namely:

\begin{conjecture}
For $k \geq 2$ there is a constant $g(k)$  such that $g(\Delta,k)  \leq g(k) \sqrt{\Delta}$.  
\end{conjecture}

This conjecture, if true,  implies the Caro-Yuster conjecture.    }
\item{We introduced the notion of a $k$-feasible graph and the corresponding function $h(\Delta,k)$ discussed in Section 4.  We have determined the exact values of $h(\Delta,k)$  for $\Delta = 0,1,2$.  We pose the problem to determine more exact values of $h(\Delta,k)$  in particular for $\Delta = 3$ as well as to determine $h(k) = \max\{ h(\Delta,k) :   \Delta \geq 0\}$.  Clearly as already proved in section 4,  $h(k)  \leq R(k,k) -1$.}
\item{Lastly we mention again the conjecture about forests:
 \begin{conjecture}
If  $F$  is a forest on $n$ vertices, where   $n \leq \frac{t^3 + 6t^2 +17t +12}{6}$  then   $f(F) \leq t$  and this bound is sharp.
\end{conjecture}}
\end{enumerate}

\bibliographystyle{plain}
\bibliography{2maxbibnew5}

\begin{thebibliography}{10}

\bibitem{albertson79con}
M.O. Albertson and D.M. Berman.
\newblock A conjecture on planar graphs. {G}raph {T}heory and {R}elated
  {T}opics ({J}.{A}. {B}ondy and {U}.{S}.{R}. {M}urty, eds.), 1979.

\bibitem{albertson1991ramsey}
M.O. Albertson and D.M. Berman.
\newblock Ramsey graphs without repeated degrees.
\newblock {\em Congressus Numerantium}, 83:91--96, 1991.

\bibitem{albertson1994lower}
M.O. Albertson and D.L. Boutin.
\newblock Lower bounds for constant degree independent sets.
\newblock {\em Discrete Mathematics}, 127(1-3):15--21, 1994.

\bibitem{balister2016random}
P.~Balister, B.~Bollob{\'a}s, J.~Lehel, and M.~Morayne.
\newblock Random hypergraph irregularity.
\newblock {\em SIAM Journal on Discrete Mathematics}, 30(1):465--473, 2016.

\bibitem{bollobas1996degree}
B.~Bollob{\'a}s.
\newblock Degree multiplicities and independent sets in ${K}_4$-free graphs.
\newblock {\em Discrete Mathematics}, 158(1-3):27--35, 1996.

\bibitem{bollobas2004extremal}
B.~Bollob{\'a}s.
\newblock {\em Extremal Graph Theory}.
\newblock Courier Corporation, 2004.

\bibitem{bollobas1997independent}
B.~Bollob{\'a}s and A.D. Scott.
\newblock Independent sets and repeated degrees.
\newblock {\em Discrete Mathematics}, 170(1-3):41--49, 1997.

\bibitem{caro2016regular}
Y.~Caro, A.~Hansberg, and R.~Pepper.
\newblock Regular independent sets.
\newblock {\em Discrete Applied Mathematics}, 203:35--46, 2016.

\bibitem{caro2014forcing}
Y.~Caro, A.~Shapira, and R.~Yuster.
\newblock Forcing $ k $-repetitions in degree sequences.
\newblock {\em The Electronic Journal of Combinatorics}, 21(1):P1--24, 2014.

\bibitem{caro2009repetition}
Y.~Caro and D.B. West.
\newblock Repetition number of graphs.
\newblock {\em The Electronic Journal of Combinatorics}, 16(1):R7, 2009.

\bibitem{caro2010large}
Y.~Caro and R.~Yuster.
\newblock Large induced subgraphs with equated maximum degree.
\newblock {\em Discrete Mathematics}, 310(4):742--747, 2010.

\bibitem{erdos1993ramsey}
P.~Erd{\H{o}}s, G.~Chen, C.C. Rousseau, and R.H. Schelp.
\newblock Ramsey problems involving degrees in edge-colored complete graphs of
  vertices belonging to monochromatic subgraphs.
\newblock {\em European Journal of Combinatorics}, 14(3):183--189, 1993.

\bibitem{erdHos1995degree}
P.~Erd{\H{o}}s, R.~Faudree, T.J. Reid, R.~Schelp, and W.~Staton.
\newblock Degree sequence and independence in ${K}_4$-free graphs.
\newblock {\em Discrete Mathematics}, 141(1-3):285--290, 1995.

\bibitem{gyarfas1992irregularity}
A.~Gy{\'a}rf{\'a}s, M.~Jacobson, L.~Kinch, J.~Lehel, and R.~Schelp.
\newblock Irregularity strength of uniform hypergraphs.
\newblock {\em J. Comb. Methods Comb. Comput}, 11:161--172, 1992.

\bibitem{ore1962theory}
O.~Ore.
\newblock {\em Theory of Graphs}, volume~38.
\newblock American Mathematical Society Providence, RI, 1962.

\bibitem{westintro}
D.B. West.
\newblock {\em Introduction to Graph Theory}.
\newblock Prentice Hall, 3rd edition, 2017.

\end{thebibliography}
\end{document}